\documentclass{amsart}
\usepackage{stmaryrd}
\usepackage{silence}
\usepackage{upgreek}
\usepackage{amssymb}
\usepackage{faktor}
\usepackage[all,cmtip]{xy}
\usepackage{amsmath} 
\usepackage{mathrsfs}
\usepackage{tikz}
\usepackage{tikz-cd}
\usepackage{enumitem}
\usepackage{mathtools}
\usepackage[mathcal]{euscript}
\usepackage{graphicx}
\usepackage{url}
\usepackage{hyperref}
\usepackage[T1]{fontenc}
\usepackage{subfigure}
\usepackage{fancyhdr}
\usetikzlibrary{shapes.geometric}
\usetikzlibrary{decorations.markings}
\usetikzlibrary{patterns,decorations.pathreplacing}

\hypersetup{colorlinks=false,linkcolor=blue}

\usepackage{ragged2e}

\newcommand{\Conf}{\mathrm{Conf}}

\DeclareMathOperator*{\im}{\mathrm{im}}

\definecolor{coloryellow}{RGB}{240,228,66}
\definecolor{colorskyblue}{RGB}{86,180,233}
\definecolor{colorvermillion}{RGB}{213,94,0}

\newcommand{\graphfont}{\mathsf}

\newcommand{\thetagraph}[1]{\graphfont{\Theta}_{#1}}
\newcommand{\completegraph}[1]{\graphfont{K}_{#1}}

\newcommand{\stargraph}[1]{\graphfont{S}_{#1}}
\newcommand{\graf}{\graphfont{\Gamma}}

\newcommand{\lollipopgraph}[1]{\graphfont{L}_{#1}}

\DeclareSymbolFont{sfletters}{OT1}{cmss}{m}{n}
\DeclareMathSymbol{\sTheta}{\mathord}{sfletters}{"02}

\theoremstyle{definition}
\newtheorem{definition}{Definition}[section]

\newtheorem{example}[definition]{Example}

\theoremstyle{plain}
\newtheorem{proposition}[definition]{Proposition}
\newtheorem{lemma}[definition]{Lemma}
\newtheorem{corollary}[definition]{Corollary}
\newtheorem{theorem}[definition]{Theorem}

\theoremstyle{remark}
\newtheorem{remark}[definition]{Remark}

\makeatletter
\@addtoreset{definition}{section}
\makeatother

\usepackage{marginnote}
    \DeclareFontFamily{U}{wncy}{}
    \DeclareFontShape{U}{wncy}{m}{n}{<->wncyr10}{}
    \DeclareSymbolFont{mcy}{U}{wncy}{m}{n}
    \DeclareMathSymbol{\Sha}{\mathord}{mcy}{"58}

\newsavebox{\foobox}

\title{Farber's conjecture for planar graphs}
\author{Ben Knudsen}
\email{b.knudsen@northeastern.edu}
\address{Department of Mathematics, Northeastern University, Boston, MA 02115, USA}

\begin{document}
\maketitle

\begin{abstract}
We prove that the ordered configuration spaces of planar graphs have the highest possible topological complexity generically, as predicted by a conjecture of Farber. Our argument establishes the same generic maximality for all higher topological complexities. We include some discussion of the non-planar case, demonstrating that the standard approach to the conjecture fails at a fundamental level.
\end{abstract}

\section{Introduction}

The problem of multiple occupancy in a space $X$ is captured by the topology of the configuration spaces \[\Conf_k(X)=\{(x_1,\ldots, x_k)\in X^k: x_i\neq x_j \text{ if } i \neq j\}.\] The problem of movement within $X$ is reflected in its topological complexity $\mathrm{TC}(X)$, a numerical invariant introduced by Farber \cite{Farber:TCMP}, whose magnitude reflects the constraints in motion planning imposed by the shape of $X$.

Combining these ideas, one is led to study the topological complexity of configuration spaces as a measure of the difficulty of collision-free motion planning \cite{FarberGrant:TCCS, CohenFarber:TCCFMPS}. On the edge of a future of automated factories and autonomous vehicles, graphs form a natural class of background spaces in which to study this problem \cite{Ghrist:CSBGGR}. Configuration spaces of graphs have been the subject of considerable recent research, including partial results on their topological complexity \cite{Aguilar-GuzmanGonzalezHoekstra-Mendoza:FSMTMHTCOCST, Farber:CFMPG,Farber:CSRMPA,LuetgehetmannRecio-Mitter:TCCSFAGBG, Scheirer:TCUCSCG}. Nevertheless, a definitive calculation has remained conjectural for 15 years \cite[\S9]{Farber:CFMPG}.

The aim of this paper is to perform this calculation in the planar case; in fact, we will calculate all of the ``higher'' topological complexities $\mathrm{TC}_r$ as well \cite{Rudyak:HATC}. Write $m(\graf)$ for the number of essential (valence at least $3$) vertices of the graph $\graf$.

\begin{theorem}\label{thm:farber}
Let $\graf$ be a connected planar graph with $m(\graf)\geq2$. For $k\geq 2m(\graf)$, we have the equality \[\mathrm{TC}_r(\Conf_k(\graf))=rm(\graf).\]
\end{theorem}

The value $rm(\graf)$ is an upper bound for dimensional reasons, so the result should be read as saying that the maximum value is achieved generically. This generic maximality or stability is an example of a more widespread phenomenon awaiting systematic explanation.

The proof proceeds by establishing a general lower bound in Theorem \ref{thm:main}, which coincides with the upper bound in the range of interest. For this result, we adapt a now standard cohomological argument of \cite{Farber:CFMPG} (see also \cite{Aguilar-GuzmanGonzalezHoekstra-Mendoza:FSMTMHTCOCST,LuetgehetmannRecio-Mitter:TCCSFAGBG}). Our innovation is to work instead with the cohomology of the configuration spaces of the plane via a planar embedding. In this way, we circumvent the difficulty posed by our relative lack of knowledge of the cohomology of $\Conf_k(\graf)$ for general $\graf$.

Farber's conjecture is that Theorem \ref{thm:farber} holds in the case $r=2$ without the assumption of planarity. We discuss the non-planar case in Section \ref{section:non-planar graphs}, proving that the standard cohomological argument cannot possibly succeed (Theorem \ref{thm:non-planar}). We believe that genuinely new methods are required to (dis)prove the non-planar case of the conjecture.

\subsection{Conventions} Unless otherwise specified, (co)homology is implicitly taken with coefficients in $\mathbb{F}_2$. A graph is a finite CW complex of dimension $1$. Given an injective continuous map $\varphi:X \to Y$, we abuse the letter $\varphi$ in using it again to denote the induced map on configuration spaces.

\subsection{Acknowledgements} The author thanks Byung Hee An and Jesus Gonzalez for helpful conversations. Special thanks are due to Andrea Bianchi, who discovered the error in an incorrect proof of the general case of Farber's conjecture appearing in an earlier version. The author learned of Farber's conjecture at the AIM workshop ``Configuration spaces of graphs'' and was reminded of it at the BIRS--CMO workshop ``Topological complexity and motion planning.'' While writing, the author benefited from the hospitality of the MPIM and was supported by NSF grant DMS 1906174.

\section{Topological complexity}

Although we will not use it directly, we give the definition of topological complexity for the sake of completeness.

\begin{definition}
The $r$th \emph{topological complexity} of a space $X$, denoted $\mathrm{TC}_r(X)$, is one more than the minimal cardinality of an open cover $\mathcal{U}$ of $X$ with the property that the map \begin{align*}X^{[0,1]}&\longrightarrow X^r\\
\gamma&\mapsto \left(\gamma\left(\frac{0}{r-1}\right), \gamma\left(\frac{1}{r-1}\right),\ldots, \gamma\left(\frac{r-1}{r-1}\right)\right)
\end{align*} admits a local section over each member of $\mathcal{U}$.
\end{definition}

\begin{remark}
Our convention is that the topological complexity of a point is $0$ rather than $1$. Both choices are common in the literature.
\end{remark}

In this paper, we interact with topological complexity exclusively through two inequalities. Before stating them, we recall the following definition. As usual, we denote diagonal maps of spaces generically by $\Delta$.

\begin{definition}
Fix a topological space $X$. An element $\zeta\in H^*(X^r)$ is called an $r$-fold \emph{zero-divisor} if $\Delta^*\zeta=0$. The $r$th \emph{zero-divisor cup length} of $X$, denoted $\mathrm{zcl}_r(X)$, is the maximal cardinality of a set of $r$-fold zero-divisors whose cup product is nonzero.
\end{definition}

In view of the K\"{u}nneth isomorphism, we usually fail to distinguish between $H^*(X^r)$ and $H^*(X)^{\otimes r}$.

\begin{example}\label{example:bar notation}
Given $\alpha\in H^*(X)$, the element $\zeta(\alpha):=\alpha\otimes 1+1\otimes\alpha$ is a zero-divisor. More generally, given $1\leq a<b\leq r$, the element $\zeta^{ab}(\alpha):=\pi_{ab}^*(\bar \alpha)$ is an $r$-fold zero-divisor, where $\pi_{ab}:X^r\to X^2$ is the projection onto the $a$th and $b$th factors. Note that the parameter $r$ is implicit in this notation.
\end{example}

We have the following well-known result, which we do not state in the greatest possible generality. 

\begin{theorem}[{\cite{Farber:TCMP,BasabeGonzalezRudyakTamaki:HTCS}}]\label{thm:tc inequalities}
If $X$ has the homotopy type of a connected CW complex of dimension $m$, then \[\mathrm{zcl}_r(X)\leq \mathrm{TC}_r(X)\leq rm.\]
\end{theorem}

We close this section with a simple device for interacting with zero-divisors.

\begin{definition}
A \emph{witness} for the cohomology class $\alpha\in H^*(A)$ is a map $f$ with target $A$ such that $f^*\alpha\neq0$.
\end{definition}

The following result allows inequalities involving zero-divisor cup lengths to be moved between spaces.

\begin{lemma}\label{lem:witness}
Let $\zeta_1,\ldots, \zeta_s\in H^*(Y^r)$ be $r$-fold zero-divisors and $f:X^r\to Y^r$ a continuous map. If a witness for the product $\zeta=\zeta_1\cdots\zeta_s$ factors through $f$, then $\mathrm{zcl}_r(X)\geq s$.
\end{lemma}
\begin{proof}
If $f\circ g$ is a witness for $\zeta$, then $g$ is a witness for $f^*\zeta=\prod_{i=1}^s f^*\zeta_i$. Since $f^*\zeta_i$ is an $r$-fold zero-divisor for each $i$ by naturality of the diagonal, the claim follows.
\end{proof}

\section{Euclidean configuration spaces}

In this section, we recall some standard facts pertaining to the cohomology of the ordered configuration spaces of $\mathbb{R}^n$, where $n>1$. Although the language differs, our perspective is heavily influenced by the beautiful and systematic treatment of \cite{Sinha:HLDO}. 

We begin by recalling that, for each $1\leq i\neq j\leq k$, one has the Gauss map \[\gamma_{ij}:\Conf_k(\mathbb{R}^n)\to S^{n-1}\] sending the configuration $(x_1,\ldots, x_k)$ to the unit vector from $x_i$ to $x_j$. We obtain a class $\alpha_{ij}\in H^{n-1}(\Conf_k(\mathbb{R}^n))$ by pulling the fundamental class of $S^{n-1}$ back along $\gamma_{ij}$. Note that our notation does not reflect the dependence on $k$ (or on $n$).

In the case $k=2$, the map $\gamma:=\gamma_{12}$ is a homotopy equivalence. More specifically, considering the antipodal embedding \begin{align*}
S^{n-1}&\xrightarrow{\iota} \Conf_2(\mathbb{R}^n)\\
x&\mapsto (-x,x),
\end{align*} we have the following standard result.

\begin{lemma}\label{lem:two points}
The composite $\gamma\circ\iota$ is the identity, and the composite $\iota\circ\gamma$ is homotopic to the identity.
\end{lemma}

We now introduce a combinatorial notion for organizing constructions of more elaborate (co)homology classes.

\begin{definition}
Fix a finite ground set $S$. A \emph{partial binary cover} of $S$ is a finite collection $\lambda$ of subsets of $S$ of cardinality $2$. We say that $\lambda$ is a \emph{binary cover} if the union of its members is $S$, a \emph{partial binary partition} if its members are disjoint, and a \emph{binary partition} if both of these conditions hold. 
\end{definition}

Given a partial binary cover $\lambda$ of $\{1,\ldots, k\}$, we obtain the cohomology class \[\alpha_\lambda=\prod_{\{i,j\}\in \lambda}\alpha_{ij}\in H^{|\lambda|(n-1)}(\Conf_k(\mathbb{R}^n)).\] Note that the product is well-defined since we work over $\mathbb{F}_2$. If $\lambda$ is a partial binary partition, then we also have the homology class $\tau_\lambda$ given by the image of the fundamental class of the torus
\[\varphi_\lambda:(S^{n-1})^\lambda \xrightarrow{\iota^\lambda} \Conf_2(\mathbb{R}^n)^\lambda\subseteq \Conf_2(\mathbb{R}^n)^\lambda\times\Conf_1(\mathbb{R}^n)^{k-2|\lambda|}\to \Conf_k(\mathbb{R}^n).\] Here, the inclusion is induced by the inclusion of any point in the second factor, and the second map is induced by any choice of $|\lambda|+k-2|\lambda|$ orientation preserving self-embeddings of $\mathbb{R}^n$ with pairwise disjoint images. The space of such embeddings is connected, so $\varphi_\lambda$ is well-defined up to homotopy and $\tau_\lambda$ therefore well-defined.

\begin{lemma}\label{lem:projection or null}
Let $\lambda$ be a partial binary partition. 
\begin{enumerate}
\item If $\{i,j\}\in \lambda$, then $\gamma_{ij}\circ \varphi_\lambda$ is homotopic to projection onto the corresponding factor of $(S^{n-1})^\lambda$.  
\item If $\{i,j\}\notin \lambda$, then $\gamma_{ij}\circ \varphi_\lambda$ is nullhomotopic.
\end{enumerate}
\end{lemma}
\begin{proof}
Supposing that $\{i,j\}\in \lambda$, we may choose the embeddings used in its definition so as to guarantee that $\varphi_\lambda$ fits into the commuting diagram \[\xymatrix{
(S^{n-1})^\lambda\ar[d]\ar[r]^{\varphi_\lambda}&\Conf_k(\mathbb{R}^n)\ar[d]^-\pi\ar[r]^-{\gamma_{ij}}&S^{n-1}\\
S^{n-1}\ar[r]^-\iota&\Conf_2(\mathbb{R}^n)\ar[r]^-{\gamma}&S^{n-1},\ar@{=}[u]
}\] where $\pi$ denotes the coordinate projection. The claim now follows from Lemma \ref{lem:two points}. 

Supposing that $\{i,j\}\notin \lambda$, we may choose the embeddings to guarantee the existence of a single hyperplane separating $x_i$ from $x_j$ for every $(x_1,\ldots, x_k)\in \im(\varphi_\lambda)$. It follows that $\gamma_{ij}\circ\varphi_\lambda$ is not surjective, implying the claim.
\end{proof}

We write $\delta$ for the Kronecker delta function and $\langle-,-\rangle$ for the Kronecker pairing of cohomology and homology.

\begin{lemma}\label{lem:kronecker}
Given a partial binary partition $\lambda$ and a partial binary cover $\mu$ of $\{1,\ldots, k\}$, \[\langle \alpha_\mu,\tau_\lambda\rangle=\delta(\lambda,\mu).\]
\end{lemma}
\begin{proof}
Note that $\alpha_\lambda$ is obtained from the fundamental class of $(S^{n-1})^\lambda$ by pullback along the second map in the composite \[(S^{n-1})^\lambda\xrightarrow{\varphi_\lambda} \Conf_k(\mathbb{R}^n)\xrightarrow{(\gamma_{ij})_{\{i,j\}\in\lambda}} (S^{n-1})^\lambda.\] If $\lambda=\mu$, then Lemma \ref{lem:projection or null} shows that this composite is homotopic to the identity. If $\lambda\neq \mu$, then the same lemma shows that this composite factors up to homotopy through a subtorus of positive codimension. In either case, the claim follows. 
\end{proof}

\section{Orthogonality and zero-divisors}

In this section, we construct some nonvanishing products of zero-divisors. We largely follow \cite{Aguilar-GuzmanGonzalezHoekstra-Mendoza:FSMTMHTCOCST}, although our notation differs.

\begin{definition}
We say that partial binary covers $\lambda_1$ and $\lambda_2$ are \emph{orthogonal} if $\lambda_1\cap\lambda_2=\varnothing$.
\end{definition}

Given a partial binary partition $\lambda$ of $\{1,\ldots, k\}$, and fixing $r$ and $1\leq a<b\leq r$, we have the following $r$-fold zero-divisor (see Example \ref{example:bar notation}):

\[\zeta^{ab}_\lambda=\prod_{\{i,j\}\in \lambda}\zeta^{ab}(\alpha_{ij})\in H^*(\Conf_{k}(\mathbb{R}^n))^{\otimes r}.\] Alternatively, we have the formula

\[\zeta_{\lambda}^{ab}=\sum_{\mu\subseteq\lambda}1\otimes\cdots\otimes \alpha_\mu\otimes\cdots\otimes\alpha_{\mu^c}\otimes\cdots\otimes 1,\] with $\alpha_\mu$ and $\alpha_{\mu^c}$ appearing in the $a$th and $b$th factor, respectively.

\begin{lemma}\label{lem:orthogonal}
Suppose given partial binary partitions $\lambda_j$ for $1\leq j\leq r$. If $\lambda_1 \perp \lambda_2$, then \[\left\langle\zeta_{\lambda_1}^{12}\zeta_{\lambda_2}^{12}\prod_{j=3}^r\zeta_{\lambda_j}^{(j-1)j},\, \bigotimes_{j=1}^r \tau_{\lambda_j}\right\rangle=1.\]
\end{lemma}
\begin{proof}
We first consider the case $r=2$, which is the claim $\left\langle\zeta_{\lambda_1}\zeta_{\lambda_2},\tau_{\lambda_1}\otimes\tau_{\lambda_2}\right\rangle=1$.
Since $\alpha_{ij}^2=0$, this inner product is a sum of terms of the form $\langle\alpha_{\mu_1}\otimes\alpha_{\mu_2},\tau_{\lambda_1}\otimes\tau_{\lambda_2}\rangle$, where $\mu_1$ and $\mu_2$ are partial binary covers. By Lemma \ref{lem:kronecker}, such a term is equal to $\delta(\lambda_1,\mu_1)\delta(\lambda_2,\mu_2)$. Since $\lambda_1\perp\lambda_2$, there is precisely one term with $\lambda_1=\mu_1$ and $\lambda_2=\mu_2$, and the claim follows. 

In the general case, we note that all but the first two factors of the product class in question coincide with those of \begin{align*}
\prod_{j=3}^r\zeta_{\lambda_j}^{(j-1)j}
&=\sum_{\mu_j\subseteq\lambda_j,\,3\leq j\leq r}
1\otimes\alpha_{\mu_3}\otimes\alpha_{\mu_3^c}\alpha_{\mu_4}\otimes\cdots\otimes \alpha_{\mu_{r-1}^c}\alpha_{\mu_{r}}\otimes\alpha_{\mu_{r}^c}.
\end{align*} For degree reasons, the only terms in this expression not annihilated by evaluation on $1\otimes 1\otimes \tau_{\lambda_3}\otimes\cdots\otimes \tau_{\lambda_r}$ are those with $\mu_{j}=\varnothing$ for every $j$, whence \begin{align*}\left\langle \zeta_{\lambda_1}^{12}\zeta_{\lambda_2}^{12}\prod_{j=3}^r\zeta_{\lambda_j}^{(j-1)j},\,\bigotimes_{j=1}^r \tau_{\lambda_j}\right\rangle&=\left\langle\zeta_{\lambda_1}\zeta_{\lambda_2},\tau_{\lambda_1}\otimes\tau_{\lambda_2}\right\rangle\cdot \prod_{j=3}^r\left\langle \alpha_{\lambda_j},\tau_{\lambda_j}\right\rangle\\
&=\left\langle\zeta_{\lambda_1}\zeta_{\lambda_2},\tau_{\lambda_1}\otimes\tau_{\lambda_2}\right\rangle\\
&=1
\end{align*} by Lemma \ref{lem:kronecker} and the previous case.
\end{proof}

We close this section by recording the following simple observation.

\begin{lemma}\label{lem:orthogonal exist}
For any $d>1$, there exists an orthogonal pair of binary partitions of $\{1,\ldots, 2d\}$
\end{lemma}
\begin{proof}
The sets $\left\{\{1,2\},\{3,4\},\ldots, \{2d-1,2d\}\right\}$ and $\left\{\{2,3\},\{4,5\},\ldots, \{2d,1\}\right\}$ are both binary partitions, which are orthogonal for $d>1$ (the two are equal for $d=1$).
\end{proof}

\section{Stars and circles}\label{section:stars and circles}

We write $\stargraph{3}$ for the \emph{star graph} with three edges, which is the cone on the discrete space $\{1,2,3\}$ with its canonical cell structure. It will be convenient to work with the coordinates \[\stargraph{3}=\left\{(t_1, t_2, t_3)\in [0,1]^3\mid \#\{i: t_i\neq 0\}\leq 1\right\}.\] 

We have the following standard result concerning the configuration space of two points in this graph---see \cite{Abrams:CSBGG}, for example. Write $\epsilon$ for the sixfold concatenated path \[\epsilon=(e_1, \bar e_2)\star (\bar e_1, e_3)\star (e_2, \bar e_3) \star (\bar e_2, e_1)\star (e_3, \bar e_1)\star (\bar e_3, e_2),\] where $e_i:[0,1]\to \stargraph{3}$ is the unique path with $t_i=e_i(t)$, and $\bar e_i$ is its reverse.

\begin{lemma}\label{lem:star is circle}
The subspace $\im (\epsilon)\subseteq \Conf_2(\stargraph{3})$ is a topological circle and a deformation retract.
\end{lemma}

The parametrization $\epsilon$ supplies an orientation, which is tied to the canonical cyclic ordering of the edges of $\stargraph{3}$. Via the standard orientation of $\mathbb{R}^2$, any piecewise smooth embedding $\varphi:\stargraph{3}\to \mathbb{R}^2$ induces a second cyclic ordering on this set.

\begin{definition}
Let $\varphi:\stargraph{3}\to \mathbb{R}^2$ be a piecewise smooth embedding. 
\begin{enumerate}
\item We say that $\varphi$ is \emph{orientation preserving} if it induces the canonical cyclic ordering on the edges of $\stargraph{3}$. 
\item We say that an orientation preserving, piecewise linear embedding $\varphi$ is \emph{standard} if the vectors $\varphi(1,0,0)$, $\varphi(0,1,0)$, and $\varphi(0,0,1)$ are pairwise linearly independent.
\end{enumerate}
\end{definition}

We now state the main result of this section.

\begin{proposition}\label{prop:star is euclidean}
For any orientation preserving, piecewise smooth embedding $\varphi:\stargraph{3}\to \mathbb{R}^2$, the composite map \[S^1\xrightarrow{\epsilon} \Conf_2(\stargraph{3})\xrightarrow{\varphi}\Conf_2(\mathbb{R}^2)\xrightarrow{\gamma} S^1\] has degree $1$.
\end{proposition}

\begin{lemma}\label{lem:piecewise smooth diagram}
Given an orientation preserving, piecewise smooth embedding $\varphi:\stargraph{3}\to \mathbb{R}^2$, there is a diagram \[\xymatrix{
\stargraph{3}\ar[d]\ar[r]&\mathbb{R}^2\ar[d]\\
\stargraph{3}\ar[r]^-\varphi&\mathbb{R}^2
}\] commuting up to piecewise smooth isotopy fixing the $0$-skeleton of $\stargraph{3}$, in which the righthand map is an orientation preserving embedding, the lefthand map the embedding of a piecewise linear isotopy retract, and the top map a standard embedding. 
\end{lemma}
\begin{proof}
In a sufficiently small geodesic neighborhood of the image of the essential vertex of $\stargraph{3}$, the map $\varphi$ is isotopic to the piecewise linear embedding given by the respective one-sided derivatives, which is orientation preserving since $\varphi$ was. If necessary, a second isotopy guarantees that this piecewise linear embedding is standard.
\end{proof}

\begin{proof}[Proof of Proposition \ref{prop:star is euclidean}] If $\varphi$ is standard, then the composite in question is bijective, hence a homeomorphism, and it is easily checked to be orientation preserving. In the general case, we appeal to Lemma \ref{lem:piecewise smooth diagram} to obtain the homotopy commutative middle square in the diagram
\[\xymatrix{
&\Conf_2(\stargraph{3})\ar[dd]\ar[r]&\Conf_2(\mathbb{R}^2)\ar[dd]\ar[dr]^-\gamma\\
S^1\ar[ur]^-\epsilon\ar[dr]^-\epsilon&&&S^1\\
&\Conf_2(\stargraph{3})\ar[r]^-\varphi&\Conf_2(\mathbb{R}^2)\ar[ur]^-\gamma. 
}\] Since each vertical map is homotopic to the appropriate identity map, the two triangles are also homotopy commutative, so the upper and lower total composites are homotopic. By the previous case, the upper total composite has degree $1$, so the claim follows by homotopy invariance of the degree.
\end{proof}

\section{A diagram}

Fixing a set $W$ of essential vertices of the planar graph $\graf$ of cardinality $d$, a binary partitition $\lambda$ of $\{1,\ldots, 2d\}$, and a bijection $W\cong \lambda$, we explain the construction of the commuting diagram

{
\[\xymatrix{
\Conf_2(\stargraph{3})^{W}\ar[dr]\ar[dd]_-{\prod_{v\in W}\psi_v}\ar[r]^-{(4)}&\Conf_2(\stargraph{3})^W\times\Conf_{k-2d}(\mathbb{R})\ar[r]^-{(3)}&\Conf_k(\graf)\ar[d]^-{}\\
\ar[d]&\prod_{v\in W}\Conf_2(U_v)\ar[dr]^-{(2)}&\Conf_{k}(\mathbb{R}^2)\ar[d]^-{\pi}\\
\Conf_2(\mathbb{R}^2)^{W}\ar[ur]_-{\cong}^-{(1)}\ar[rr]&&\Conf_{2d}(\mathbb{R}^2),
}\]} where $\pi$ denotes the coordinate projection.

We begin by choosing a piecewise smooth embedding $\graf\subseteq \mathbb{R}^2$. Next, for each $v\in W$, we choose a coordinate neighborhood $v\in U_v\subseteq \mathbb{R}^2$ with the property that $U_v\cap \graf$ is connected and contains no vertex of $\graf$ other than $v$. We further require the $U_v$ to be pairwise disjoint. The first numbered map is determined by choosing an orientation preserving diffeomorphism $\mathbb{R}^2\cong U_v$ for each $v\in W$. The second numbered map is determined by the inclusions $U_v\subseteq \mathbb{R}^2$ and the bijection $W\cong\lambda$.

Finally, for each $v\in W$, we choose a piecewise linear embedding $\stargraph{3}\to\graf$ with image lying in $U_v$ (in particular, the image of the essential vertex is $v$). We further require that the composite of each such embedding with $\iota$ be orientation preserving. Separately, we choose a smooth embedding of $\mathbb{R}$ into an edge of $\graf$ with image disjoint from each $U_v$. The third numbered map is determined by these embeddings and the bijection $W\cong\lambda$, and the fourth numbered map is determined by choosing any point in $\Conf_{k-2d}(\mathbb{R})$.

The remainder of the diagram is determined by commutativity. In particular, we obtain the maps $\psi_v$, concerning which we have the following immediate consequence of Proposition \ref{prop:star is euclidean}.

\begin{corollary}\label{cor:degree}
For each $v\in W$, the composite \[S^1\to \Conf_2(\stargraph{3})\xrightarrow{\psi_v} \Conf_2(\mathbb{R}^2)\xrightarrow{\gamma} S^1\] has degree $1$.
\end{corollary}

\section{Planar graphs}

The goal of this section is to establish the following lower bound, from which Theorem \ref{thm:farber} will easily follow.

\begin{theorem}\label{thm:main}
Let $\graf$ be a connected planar graph. For any $r>0$ and $k\geq4$, we have the inequality \[\mathrm{TC}_r(\Conf_k(\graf))\geq r\min\left\{\left\lfloor\frac{k}{2}\right\rfloor, m(\graf)\right\}.\]
\end{theorem}

Given an integer $d$ such that $d\leq m(\graf)$ and $2d\leq k$, together with (not necessarily distinct) binary partitions $\lambda_j$ of $\{1,\ldots, 2d\}$ for $1\leq j\leq r$, the construction of the previous section provides the diagram

\[\xymatrix{
&&\displaystyle\prod_{j=1}^r\Conf_2(\stargraph{3})^{\lambda_j}\ar[d]\ar[rr]&&\Conf_k(\graf)^r\ar[d]\\
\displaystyle\prod_{j=1}^r(S^1)^{\lambda_j}\ar[urr]^-{\prod\epsilon^{\lambda_j}}\ar[rr]^-{\prod\iota^{\lambda_j}}\ar[drr]_-{\prod\varphi_{\lambda_j}}&&\displaystyle\prod_{j=1}^r\Conf_2(\mathbb{R}^2)^{\lambda_j}\ar[d]&&\Conf_k(\mathbb{R}^2)^r\ar[d]\\
&&\Conf_{2d}(\mathbb{R}^2)^r\ar@{=}[rr]&&\Conf_{2d}(\mathbb{R}^2)^r,
}\] in which the rectangular and lower triangular subdiagrams are commutative. Moreover, Lemmas \ref{lem:two points} and \ref{lem:star is circle} and Corollary \ref{cor:degree} imply that the upper triangular subdiagram is homotopy commutative.

\begin{proof}[Proof of Theorem \ref{thm:main}]
The cases $m(\graf)\in \{0,1\}$ are easily treated by other means, so we assume that $m(\graf)\geq 2$. In light of this inequality and the assumption that $k\geq4$, there is an integer $d$ satisfying the inequalities $1<d\leq m(\graf)$ and $2d\leq k$. It suffices to show that $\mathrm{TC}_r(\Conf_k(\graf))\geq rd$ for every such $d$. 

By Lemma \ref{lem:orthogonal exist}, we may find binary partitions $\lambda_j$ of $\{1,\ldots, 2d\}$ for $1\leq j\leq r$ with $\lambda_1\perp \lambda_2$. Considering the diagram shown above, it follows from orthogonality and Lemma \ref{lem:orthogonal} that the bottom diagonal map is a witness for the product class $\zeta^{12}_{\lambda_1}\zeta^{12}_{\lambda_2}\prod_{j=3}^r\zeta_{\lambda_j}^{(j-1)j}$. By commutativity and Lemma \ref{lem:witness}, it follows that $\mathrm{zcl}_r(\Conf_k(\graf))\geq rd$, whence $\mathrm{TC}_r(\Conf_k(\graf))\geq rd$ by Theorem \ref{thm:tc inequalities}.
\end{proof}

\begin{proof}[Proof of Theorem \ref{thm:farber}]
The assumption and Theorem \ref{thm:main} gives the lower bound. The upper bound bound is implied by Theorem \ref{thm:tc inequalities}, since $\Conf_k(\graf)$ has the homotopy type of a CW complex of dimension $m(\graf)$ for $k\geq m(\graf)$ by \cite{Swiatkowski:EHDCSG}.
\end{proof}

\section{Non-planar graphs}\label{section:non-planar graphs}

The proof of Theorem \ref{thm:farber}, as with its predecessors \cite{Aguilar-GuzmanGonzalezHoekstra-Mendoza:FSMTMHTCOCST,Farber:CFMPG,LuetgehetmannRecio-Mitter:TCCSFAGBG}, is premised on the recognition that an essential vertex and a pair of particles give rise to a circle (Lemma \ref{lem:star is circle}), whence $d$ essential vertices and $2d$ particles give rise to a torus of dimension $d$. When $d=m(\graf)$, this dimension is that of the configuration space at large, and one shows that the bounds of Theorem \ref{thm:tc inequalities} coincide by building long cup products out of cohomology classes detecting circle factors of the torus.

The following result, and Corollary \ref{cor:cohomologically planar} below, show that no such approach can possibly succeed in the non-planar case.

\begin{theorem}\label{thm:non-planar}
A graph $\graf$ is non-planar if and only if there is a topological embedding $\stargraph{3}\to \graf$ inducing the zero homomorphism on $H_1(\Conf_2(-);\mathbb{Z})$.
\end{theorem}

Theorem \ref{thm:non-planar} is more or less implicit in the work of Ko--Park \cite{KoPark:CGBG}, but we believe the simple argument below to be of independent value. This argument is premised on certain atomic relations in $H_1(\Conf_2(\graf);\mathbb{Z})$, which become twice the Q- and $\Theta$-relations of \cite{AnDrummond-ColeKnudsen:SSGBG} after projection to the unordered configuration space. We refer to the graphs depicted in Figure \ref{fig:graph examples}.

\begin{figure}[ht]
\begin{tikzpicture}
\begin{scope}[xshift=4cm]
\fill[black] (0,-.5) circle (2.5pt);
\fill[black] (0,-1) circle (2.5pt);
\draw(0,0) circle (.5cm);
\draw(0,-1) -- (0,-.5);
\draw(0,-1.2) node[below]{$\lollipopgraph{}$};
\end{scope}
\begin{scope}[xshift=8cm]
\fill[black] (0,-.5) circle (2.5pt);
\fill[black] (0,.5) circle (2.5pt);
\draw(0,0) circle (.5cm);
\draw(0,.5) -- (0,-.5);
\draw(0,-1.2) node[below]{$\thetagraph{3}$};
\end{scope}

\end{tikzpicture}
\caption{The lollipop graph and the theta graph}\label{fig:graph examples}
\end{figure}
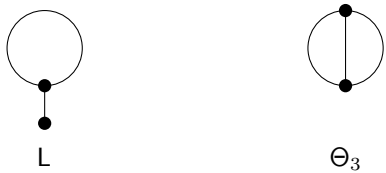

We distinguish four classes $H_1(\Conf_2(\graphfont{L});\mathbb{Z})$ (in fact, these classes span subject to the single relation below). The \emph{star class} $\sigma$ is defined as in Section \ref{section:stars and circles} by allowing the two particles to orbit one another clockwise by passing through the essential vertex; the class $\beta_i$ for $i=1,2$ is defined by allowing the $i$th particle to traverse the cycle counterclockwise while the other particle remains stationary on the leaf; and the class $\beta_{12}$ is defined by allowing both particles to traverse the cycle antipodally counterclockwise.

\begin{lemma}[2Q-relation]\label{lem:Q}
In $H_1(\Conf_2(\graphfont{L});\mathbb{Z})$, there is the relation \[\sigma=\beta_1+\beta_2-\beta_{12}.\]
\end{lemma}
\begin{proof}
By inspection, the relation holds at the chain level in the ordered {\'{S}}wi\k{a}tkowski complex---see \cite[\S 2.1]{ChettihLuetgehetmann:HCSTL}, for example.
\end{proof}

Write $\sigma_1$ for the star class in $H_1(\Conf_2(\thetagraph{3});\mathbb{Z})$ at the top vertex and $\sigma_2$ for the star class at the bottom vertex, both oriented clockwise. Applying the 2Q-relation twice, we obtain the following relation.

\begin{corollary}[2$\Theta$-relation]\label{cor:theta}
In $H_1(\Conf_2(\thetagraph{3});\mathbb{Z})$, there is the relation \[\sigma_1=\sigma_2.\]
\end{corollary}

\begin{proof}[Proof of Theorem \ref{thm:non-planar}] The ``if'' direction is an easy consequence of Proposition \ref{prop:star is euclidean}. For the converse, suppose that $\graf$ is non-planar. By Kuratowski's theorem, we may assume that $\graf$ is the complete graph $\completegraph{5}$ or the complete bipartite graph $\completegraph{3,3}$. In either case, via Corollary \ref{cor:theta}, the same argument employed in the unordered setting in \cite[Lem. C.4]{AnDrummond-ColeKnudsen:SSGBG} shows that $\sigma=-\sigma$ for any star class $\sigma$ in $\graf$. Since $H_1(\Conf_2(\graf);\mathbb{Z})$ is torsion-free \cite[Cor. 3.22]{KoPark:CGBG}, the claim follows.
\end{proof}

In what follows, a \emph{topological spanning tree} is a subspace of a graph homeomorphic to a tree and containing an open neighborhood of every essential vertex.

\begin{corollary}\label{cor:cohomologically planar}
A graph $\graf$ is planar if and only if there is a topological spanning tree $T\subseteq \graf$, a piecewise smooth embedding $\varphi:T\to \mathbb{R}^2$, and a cohomology class $\alpha^\graf_{12}\in H^1(\Conf_2(\graf))$ such that $\alpha^\graf_{12}|_{T}=\varphi^*\alpha_{12}$.
\end{corollary}
\begin{proof}
If $\graf$ admits the (without loss of generality) piecewise smooth planar embedding $\psi$, then we may choose $T\subseteq \mathbb{R}^2$ arbitrarily and set $\varphi=\psi|_T$ and $\alpha_{12}^\graf=\psi^*\alpha_{12}$. If $\graf$ is non-planar, we may consider the embedding $\stargraph{3}\to \graf$ supplied by Theorem \ref{thm:non-planar}, whose image we may take to lie in $T$ without loss of generality. Our assumptions give rise to the following commutative diagram
\[\xymatrix{
&H_1(\Conf_2(\stargraph{3});\mathbb{Z})\ar[dl]\ar[dr]^-0\\
H_1(\Conf_2(T);\mathbb{Z})\ar[d]_-\varphi\ar[rr] &&H_1(\Conf_2(\graf);\mathbb{Z})\ar[d]^-{\alpha_{12}^\graf}\\
H_1(\Conf_2(\mathbb{R}^2);\mathbb{Z})\ar[rr]^-{\alpha_{12}}&&\mathbb{F}_2.
}\] Since the counterclockwise composite is nonzero by Proposition \ref{prop:star is euclidean}, we obtain a contradiction.
\end{proof}

\bibliographystyle{plain}
\bibliography{references}

\begin{thebibliography}{10}

\bibitem{Abrams:CSBGG}
A.~Abrams.
\newblock {\em Configuration spaces of braid groups of graphs}.
\newblock PhD thesis, UC Berkeley, 2000.

\bibitem{Aguilar-GuzmanGonzalezHoekstra-Mendoza:FSMTMHTCOCST}
{Aguilar-Guzm\'{a}n, J. and Gonz\'{a}lez, J. and Hoekstra-Mendoza, T.}
\newblock Farley--sabalka's morse-theory model and the higher topological
  complexity of ordered configuration spaces of trees.
\newblock arXiv:1911.12522, 2019.

\bibitem{AnDrummond-ColeKnudsen:SSGBG}
B.~H. An, G.C. Drummond-Cole, and B.~Knudsen.
\newblock Subdivisional spaces and graph braid groups.
\newblock {\em Doc. Math.}, 24:1513--1583, 2019.

\bibitem{BasabeGonzalezRudyakTamaki:HTCS}
I.~Basabe, J.~{Gonz\'{a}lez}, Y.~Rudyak, and D.~Tamaki.
\newblock Higher topological complexity and its symmetrization.
\newblock {\em Algebr. Geom. Topol.}, 14.

\bibitem{ChettihLuetgehetmann:HCSTL}
S.~Chettih and D.~L{\"{u}}tgehetmann.
\newblock The homology of configuration spaces of trees with loops.
\newblock {\em Alg. Geom. Topol.}, 18(4):2443--2469, 2018.

\bibitem{CohenFarber:TCCFMPS}
D.~Cohen and M.~Farber.
\newblock Topological complexity of collision-free motion planning on surfaces.
\newblock {\em Compos. Math.}, 147, 2009.

\bibitem{Farber:TCMP}
M.~Farber.
\newblock Topological complexity of motion planning.
\newblock {\em Discrete Comput. Geom.}, 29:211--221, 2003.

\bibitem{Farber:CFMPG}
M.~Farber.
\newblock Collision free motion planning on graphs.
\newblock In M.~Erdmann, M.~Overmars, D.~Hsu, and F.~van~der Stappen, editors,
  {\em Algorithmic Foundations of Robotics VI}, volume~17 of {\em Springer
  Tracts Adv. Robot.}, pages 123--138. Springer, 2005.

\bibitem{Farber:CSRMPA}
M.~Farber.
\newblock Configuration spaces and robot motion planning algorithms.
\newblock {\em Lecture Notes Series, Institute for Mathematical Sciences}, 35,
  2017.

\bibitem{FarberGrant:TCCS}
M.~Farber and M.~Grant.
\newblock Topological complexity of configuration spaces.
\newblock {\em Proc. Amer. Math. Soc}, 137, 2009.

\bibitem{Ghrist:CSBGGR}
R.~Ghrist.
\newblock Configuration spaces and braid groups on graphs in robotics.
\newblock In {\em Knots, braids, and mapping class groups---papers dedicated to
  {J}oan {S}. {B}irman (New York, 1998)}, volume~24 of {\em AMS/IP Stud. Adv.
  Math.}, pages 29--40. Amer. Math. Soc., 2002.

\bibitem{KoPark:CGBG}
Ko.~K. H. and H.W. Park.
\newblock Characteristics of graph braid groups.
\newblock {\em Discrete Comput. Geom.}, 48:915--963, 2012.

\bibitem{LuetgehetmannRecio-Mitter:TCCSFAGBG}
D.~Luetgehetmann and D.~Recio-Mitter.
\newblock Topological complexity of configuration spaces of fully articulated
  graphs and banana graphs.
\newblock {\em Discrete Comput. Geom.}, 2019.

\bibitem{Rudyak:HATC}
Y.~Rudyak.
\newblock On higher analogues of topological complexity.
\newblock {\em Topology Appl.}, 157:916--920, 2010.

\bibitem{Scheirer:TCUCSCG}
S.~Scheirer.
\newblock Topological complexity of unordered configuration spaces of certain
  graphs.
\newblock {\em Topol. Appl}, 2020.

\bibitem{Sinha:HLDO}
D.~Sinha.
\newblock The homology of the little disks operad.
\newblock arXiv:0610236, 2010.

\bibitem{Swiatkowski:EHDCSG}
J.~{\'{S}}wi\k{a}tkowski.
\newblock Estimates for homological dimension of configuration spaces of
  graphs.
\newblock {\em Colloq. Math.}, 89(1):69--79, 2001.

\end{thebibliography}

\end{document}